\documentclass[reqno]{amsart}

\usepackage{mathrsfs}
\usepackage{amsmath}
\usepackage{amssymb}
\usepackage{cite}
\usepackage{latexsym}
\usepackage{graphicx}

\usepackage{color}
\usepackage{comment}

\theoremstyle{plain}
\newtheorem{thm}{Theorem}[section]
\newtheorem{lem}[thm]{Lemma}

\newtheorem{prop}[thm]{Proposition}
\newtheorem{rmk}[thm]{Remark}

\newtheorem{ex}[thm]{Example}

\def\G{\mathscr{G}}

\def\P{\mathscr{P}}

\def\T{\mathscr{T}}

\def\d{\mathrm{d}}

\def\Cset{\mathbb{C}}
\def\Kset{\mathbb{K}}
\def\Lset{\mathbb{L}}

\def\Rset{\mathbb{R}}

\def\SL{\mathrm{SL}}

\def\id{\mathrm{id}}

\def\Re{\mathrm{Re}\,}

\def\epsilon{\varepsilon}

\DeclareMathOperator{\sech}{sech}

\makeatletter
 \@addtoreset{equation}{section}
\makeatother
\def\theequation{\arabic{section}.\arabic{equation}}

\begin{document}


\title[Solvability of the Zakharov-Shabat systems by quadrature]
{Solvability of the Zakharov-Shabat systems
 with meromorphic potentials by quadrature}

\author{Kazuyuki Yagasaki}

\address{Department of Applied Mathematics and Physics, Graduate School of Informatics,
Kyoto University, Yoshida-Honmachi, Sakyo-ku, Kyoto 606-8501, JAPAN}
\email{yagasaki@amp.i.kyoto-u.ac.jp}

\date{\today}
\subjclass[2020]{35Q51, 37K15, 34M03, 34M15, 34M35, 34M40, 35P25, 37K40}
\keywords{Solvability by quadrature; Zakharov-Shabat system;
 integrable partial differential equation; inverse scattering transform;
 differential Galois theory; meromorphic potential; reflectionless potential}

\begin{abstract}
We study the solvability of the general two-dimensional Zakharov-Shabat (ZS) systems
 with meromorphic potentials by quadrature.
These systems appear in application of the inverse scattering transform (IST)
 to an important class of nonlinear partial differential equations (PDEs)
 called integrable systems.
Their solvability by quadrature is a key to obtain analytical expressions
 for solutions to the initial value problems of the integrable PDEs by using the IST.
We prove that the ZS systems are always integrable in the sense of differential Galois theory,
 i.e., solvable by quadrature,
 if and only if the meromporphic potentials are reflectionless,
 under the condition that  the potentials are absolutely integrable
 on $\Rset\setminus(-R_0,R_0)$ for some $R_0>0$.
Similar statements were previously proved to be true by the author
 for a limited class of potentials and the linear Schr\"odinger equations.
\end{abstract}
\maketitle


\section{Introduction}

In this paper we study the solvability
 of the two-dimensional \emph{Zakharov-Shabat} (ZS) \emph{systems},
\begin{equation}
w_x=
\begin{pmatrix}
-ik & q(x)\\
r(x) & ik
\end{pmatrix}w,\quad
w\in\Cset^2,
\label{eqn:ZS2}
\end{equation}
by quadrature, where $k\in\Cset$ is a constant,
 the subscript $x$ represents differentiation with respect to the variable $x$,
 and the potentials $q(x),r(x)$ are meromorphic functions
 defined in a domain including $\Rset$ in $\Cset$ such that
\begin{equation}
\lim_{x\to\pm\infty}q(x),r(x)=0.
\label{eqn:A1a}
\end{equation}
When $r(x)\equiv -1$, the ZS system \eqref{eqn:ZS2}
 reduces to the linear Schr\"odinger equation
 although $r(x)$ does not satisfy \eqref{eqn:A1a}.
See, e.g., Section~2 of \cite{Y23a}.
We also pay special attention to the case
 in which the potentials $q(x)$ and $r(x)$ are given by rational functions.
The sovability of \eqref{eqn:ZS2} by quadrature
 can be determined by the differential Galois theory \cite{CH11,PS03}.
See also Section~3 of \cite{Y23a}
 for a quick review of a necessary part of the theory.

As well-known (see, e.g., Section~1.2 of \cite{AS81}),
 the ZS system \eqref{eqn:ZS2} appears in application
 of the \emph{inverse scattering transform} (IST)
 for the following fundamental and important nonlinear
 partial differential equations (PDEs):
\begin{itemize}
\setlength{\leftskip}{-2.7em}
\item
The nonlinear Schr\"odinger (NLS) equation
\begin{equation}
iq_t=q_{xx}\pm 2|q|^2q
\label{eqn:NLS}
\end{equation}
with $r=\mp q^\ast$;
\item
The modified Korteweg-de Vries (mKdV) equation
\begin{equation}
q_t\pm 6q^2q_x+q_{xxx}=0
\label{eqn:mKdV}
\end{equation}
with $r=\mp q$;
\item
The sine-Gordon equation
\begin{equation}
u_{xt}=\sin u
\label{eqn:sineG}
\end{equation}
with $-q=r=\frac{1}{2}u_x$;
\item
The sinh-Gordon equation
\begin{equation}
u_{xt}=\sinh u
\label{eqn:sinhG}
\end{equation}
with $q=r=\frac{1}{2}u_x$.
\end{itemize}
Here $q,r$ and $u$ are assumed to depend on the time variable $t$ as well as $x$,
 and the superscript `$\ast$' represents complex conjugate.
When the minus sign is taken in \eqref{eqn:NLS} and \eqref{eqn:mKdV},
 the ZS system \eqref{eqn:ZS2} is typically considered under nonvanishing boundary conditions,
 in which $\lim_{x\to\pm\infty}q(x),r(x)=0$ does not hold.
See \cite{ZS73,O79} for more details.
So these two cases may have to be excluded in our discussions.
On the other hand,
 under the transformation $(x+t,x-t)\mapsto(x,t)$,
 the sine- and sinh-Gordon equations \eqref{eqn:sineG} and \eqref{eqn:sinhG}
 are changed to
\[
u_{tt}-u_{xx}+\sin u=0\quad\mbox{and}\quad
u_{tt}-u_{xx}+\sinh u=0,
\]
respectively, in the physical coordinate system.
The Korteweg-de Vries (KdV) equation
\begin{equation}
q_t+6qq_x+q_{xxx}=0
\label{eqn:KdV}
\end{equation}
can also be treated by the IST
 although a different type of the ZS systems or the linear Schr\"odinger equations
 need to be used (see, e.g., Chapter~9 of \cite{A11} or \cite{Y23a,Y23b}).

To solve the initial value problems (IVPs) of these PDEs using the IST,
 we need to obtain particular solutions called the \emph{Jost solutions}, which satisfy
\begin{equation}
\begin{split}
&
\phi(x;k)\sim
\begin{pmatrix}
1\\
0
\end{pmatrix}e^{-ikx},\quad
\bar{\phi}(x;k)\sim
\begin{pmatrix}
0\\
1
\end{pmatrix}e^{ikx}\quad\mbox{as $x\to-\infty$},\\
&
\psi(x;k)\sim
\begin{pmatrix}
0\\
1
\end{pmatrix}e^{ikx},\quad
\bar{\psi}(x;k)\sim
\begin{pmatrix}
1\\
0
\end{pmatrix}e^{-ikx}\quad\mbox{as $x\to+\infty$.}
\end{split}
\label{eqn:bc}
\end{equation}
to the ZS system \eqref{eqn:ZS2} for any $k\in\Cset^\ast:=\Cset\setminus\{0\}$.
Thus, the solvability of \eqref{eqn:ZS2} by quadrature
 for any $k\in\Cset^\ast$
 is a key to obtain analytical expressions for solutions to the IVPs of the integrable PDEs.

The solvability of \eqref{eqn:ZS2} by quadrature
 was recently discussed in \cite{Y23a} under the following condition:
\begin{enumerate}
\setlength{\leftskip}{-3.8mm}
\item[\bf(A${}_0$)]
The potentials $q(x),r(x)$ are holomorphic in a neighborhood $U$ of $\Rset$ in $\Cset$.
Moreover, there exist holomorphic functions $q_\pm,r_\pm:U_0\to\Cset$
 such that $q_\pm(0),r_\pm(0)=0$ and
\[
q(x)=q_\pm(e^{\mp\lambda_\pm x}),\quad
r(x)=r_\pm(e^{\mp\lambda_\pm x})
\]
for $|\Re x|$ sufficiently large,
 where $U_0$ is a neighborhood of the origin in $\Cset$,
 $\lambda_\pm\in\Cset$ are some constants with $\Re\lambda_\pm>0$,
 and the upper or lower signs are taken simultaneously
 depending on whether $\Re x>0$ or $\Re x<0$.
\end{enumerate}
In particular, $q(x),r(x)$ tend to zero exponentially as $x\to\pm\infty$
 on $\Rset$, so that $u\in L^1(\Rset)$, if they satisfy condition~(A${}_0$).
Condition~(A${}_0$) is a little restrictive,
 but it is satisfied by several wide classes of functions.
For example, if $q(x),r(x)$ are rational functions of $e^{\lambda x}$
 for some $\lambda\in\Cset$ with $\Re\lambda>0$,
 have no singularity on $\Rset$, and $q(x),r(x)\to 0$ as $x\to\pm\infty$,
 then condition~(A${}_0$) holds.
See Section~1 of \cite{Y23a} for another nontrivial example of $q(x),r(x)$
 satisfying condition~(A${}_0$).

Define the \emph{scattering coefficients}
 $a(k),\bar{a}(k),b(k),\bar{b}(k)$ for \eqref{eqn:ZS2} as
\begin{equation}
\begin{split}
\phi(x;k)=& b(k)\psi(x;k)+a(k)\bar{\psi}(x;k),\\
\bar{\phi}(x;k)=&\bar{a}(k)\psi(x;k)+\bar{b}(k)\bar{\psi}(x;k).
\end{split}
\label{eqn:ab}
\end{equation}
When $a(k),\bar{a}(k)\neq 0$, the constants
\[
\rho(k)=b(k)/a(k),\quad
\bar{\rho}(k)=\bar{b}(k)/\bar{a}(k)
\]
are defined and called the \emph{reflection coefficients} for \eqref{eqn:ZS2}.
The potentials $q(x),r(x)$ are called \emph{reflectionless}
 if $b(k),\bar{b}(k)=0$ for any $k\in\Rset^\ast:=\Rset\setminus\{0\}$.
The following result was proved in \cite{Y23a}.

\begin{thm}
\label{thm:Y}
Suppose that condition~{\rm(A${}_0$)} holds.
If the potentials $q(x),r(x)$ are reflectionless,
 then Eq.~\eqref{eqn:ZS2} is solvable by quadrature
 for any $k\in\Cset^\ast$.
Conversely, 
 if Eq.~\eqref{eqn:ZS2} is solvable by quadrature for any $k\in\Rset^\ast$,
 then the potentials $q(x),r(x)$ are reflectionless.
 \end{thm}

In this paper we consider the case in which $q(x),r(x)$ are meromorphic,
 and extend the result of \cite{Y23a}. 
More precisely, we assume the following.

\begin{enumerate}
\setlength{\leftskip}{-3.8mm}
\item[\bf(A${}_1$)]
The potentials $u=q(x),r(x)$ are meromorphic in a neighborhood of $\Rset$ in $\Cset$ and
\begin{equation}
\int_{\Rset\setminus(-R_0,R_0)}|u(x)|\d x<\infty\quad
\mbox{for some $R_0>0$.}
\label{eqn:A1}
\end{equation}
\end{enumerate}

If Eq.~\eqref{eqn:A1} holds, then so does condition~\eqref{eqn:A1a}
 and by Theorem~8.1 in Section~3.8 of \cite{CL55}
 there exist the Jost solutions $\phi(x;k),\psi(x;k)$ satisfying \eqref{eqn:bc}
 for $k\in\Cset^\ast$. 
See also Section~2 of \cite{Y23a}.
On the other hand,
 meromorphic functions that are not analytic have singular points.
Singular solitons called \emph{positons}, \emph{negatons} and \emph{complexitons}
 have attracted much attention and been studied
 for the mKdV equation \eqref{eqn:mKdV} in \cite{RSK96,S92,WZF08}
 and the sine-Gordon equation \eqref{eqn:sineG} in \cite{B93,JZ95,WF07}.
They were found to be reflectionless in many cases \cite{B93,JZ95,S92,WZF08}
 although not all.
Rational solutions have also been studied
 for the NLS equation \eqref{eqn:NLS} and the mKdV equation \eqref{eqn:mKdV}
 \cite{C06a,C06b,NH85,H97}.
See \cite{Y23b} for references of related work on the KdV equation \eqref{eqn:KdV}.

Let $\Cset_\pm=\{k\in\Cset\mid\pm\mathrm{Im}\,k>0\}$.
Our main results are stated as follows.

\begin{thm}
\label{thm:main1}
Suppose that condition~{\rm(A${}_1$)} holds
 and $a(k)$ and $\bar{a}(k)$ have zeros in $\Cset_+$ and $\Cset_-$, respectively.
If the potentials $q(x),r(x)$ are reflectionless,
 then Eq.~\eqref{eqn:ZS2} is solvable by quadrature for any $k\in\Cset^\ast$.
\end{thm}

\begin{thm}
\label{thm:main2}
Suppose that condition~{\rm(A${}_1$)} holds,
 $q(x),r(x)$ are analytic in a neighborhood of $x=\infty$
 in the Riemann sphere $\Cset\cup\{\infty\}$,
 and one of the following conditions holds$\,:$
\begin{itemize}
\setlength{\leftskip}{-1.6em}
\item[(i)]
$r(x)=q(x);$
\item[(ii)]
$r(x)=-q(x);$
\item[(iii)]
$r(x)=q(x)^\ast;$
\item[(iv)]
$r(x)=-q(x)^\ast$.
\end{itemize}
If Eq.~\eqref{eqn:ZS2} is solvable by quadrature for any $k\in\Cset^\ast$,
 then the potentials $q(x),r(x)$ are reflectionless.
\end{thm}

\begin{rmk}\
\label{rmk:1b}
\begin{itemize}
\setlength{\leftskip}{-1.6em}
\item[(i)]
The scattering coefficients $a(k),\bar{a}(k)$ may have no zero
 in $\Cset_+$ and $\Cset_-$, respectively,
 even when condition {\rm(A${}_1$)} holds and $b(k),\bar{b}(k)=0$ for any $k\in\Rset^\ast$.
In this case, Theorem~$\ref{thm:main1}$ does not apply.
\item[(ii)]
One of conditions~{\rm(i)-(iv)} in Theorem~$\ref{thm:main2}$ holds
 for the integrable PDEs \eqref{eqn:NLS}-\eqref{eqn:sinhG}.
If one of the conditions holds,
 then the number of zeros is the same for $a(k)$ and $\bar{a}(k)$. 
See Proposition~$\ref{prop:2b}$.
\end{itemize}
\end{rmk}

Results similar to Theorems~\ref{thm:main1} and \ref{thm:main2}
 were very recently obtained for the linear Schr\"odinger equation
\begin{equation}
v_{xx}+(k^2+q(x))v=0,\quad
v\in\Cset,
\label{eqn:SE}
\end{equation}
which appears in application of the IST to the KdV equation \eqref{eqn:KdV}, in \cite{Y23b}
(see also Theorem~\ref{thm:KdV} below).
Here we extend the results to the ZS system \eqref{eqn:ZS2}.
In \cite{Y23b}, rational potentials that tend to zero as $x\to\pm\infty$ but do not satisfy \eqref{eqn:A1}
 were also discussed and Eq.~\eqref{eqn:SE} was shown not to be solvable by quadrature
 for some $k\in\Rset^\ast$ in that case.
However, it is difficult to obtain such a result for the ZS system \eqref{eqn:ZS2},
 as explained in Section~5 below.

The outline of this paper is as follows.
We collect preliminary results in Section~2,
 and give proofs of Theorems~\ref{thm:main1} and \ref{thm:main2}
 in Sections~3 and 4, respectively.
Finally, we discuss rational potentials possessing the properties stated above in Section~5.


\section{Preliminary Results}
In this section we give preliminary results
 required to prove Theorems~\ref{thm:main1} and \ref{thm:main2}.
Henceforth we assume that condition~(A${}_1$) holds.

Since the trace of the coefficient matrix in \eqref{eqn:ZS2} is zero,
 we see by \eqref{eqn:bc} that the Wronskian of $\phi(x)$ and $\bar{\phi}(x)$
 (resp. of $\psi(x)$ and $\bar{\psi}(x)$) is one, i.e.,
\begin{equation}
\det(\phi(x;k),\bar{\phi}(x;k))=\det(\bar{\psi}(x;k),\psi(x;k))=1.
\label{eqn:W}
\end{equation}
Hence, it follows from \eqref{eqn:ab} that
\begin{equation}
a(k)\bar{a}(k)-b(k)\bar{b}(k)=1.
\label{eqn:2b}
\end{equation}
We have the following properties of the scattering coefficients. 

\begin{prop}\
\label{prop:2a}
\begin{itemize}
\setlength{\leftskip}{-1.6em}
\item[(i)]
$a(k),\bar{a}(k),b(k),\bar{b}(k)$ are analytic in $\Rset^\ast;$
\item[(ii)]
$a(k)$ and $\bar{a}(k)$ can be analytically continued
 in $\Cset_+$ and $\Cset_-$, respectively$\,;$
\item[(iii)]
$a(k)$ and $\bar{a}(k)$ only have finitely many zeros
 in $\Cset_+\cup\Rset$ and $\Cset_-\cup\Rset$, respectively.
\end{itemize}
\end{prop}

Similar results are obtained in Section~4.1 (especially Proposition~4.1) of \cite{Y23a}
 and Section~2 (especially Proposition~2.1) of \cite{Y23b}, respectively,
 when $q(x),r(x)$ are analytic on $\Rset$
 and when the ZS system \eqref{eqn:ZS2} reduces to the linear Schr\"odinger equations,
 i.e., $r(x)\equiv -1$.

\begin{proof}
Let $S$ denote the set of poles of $q(x)$ and $r(x)$.
Then $S$ is discrete and contains no accumulation point.
Choose a curve $\Gamma=\{\gamma(\xi)\in\Cset_+\mid \xi\in\Rset\}$ with
\[
\gamma(\xi)=\xi+ic\sech\xi
\]
does not intersect $S$.
By (A${}_1$) $q(x),r(x)$ are analytic on $\Gamma$.

Regard \eqref{eqn:ZS2} as a differential equation on $\Gamma$
 and let $U$ be a neighborhood of $\Gamma$ in $\Cset_+$ such that $S\cap U=\emptyset$.
Then its solutions $v=\phi(x;k)$ and $\psi(x;k)$ are bounded and analytic
 in $k\in\Rset^\ast$ as well as in $x\in U$.
By \eqref{eqn:ab} and \eqref{eqn:W} we have
\begin{equation}
\begin{split}
&
a(k)=\det(\phi(x;k),\psi(x;k)),\quad
\bar{a}(k)=\det(\bar{\psi}(x;k), \bar{\phi}(x;k)),\\
&
b(k)=\det(\bar{\psi}(x;k),\phi(x;k)),\quad
\bar{b}(k)=\det(\bar{\phi}(x;k),\psi(x;k)).
\end{split}
\label{eqn:prop2a1}
\end{equation}
Hence, we obtain part~(i).
Similarly we see that $\phi(x;k)$ and $\psi(x;k)$
 (resp. $\bar{\phi}(x;k)$ and $\bar{\psi}(x;k)$) are bounded and analytic
 in $\Cset_+$ (resp, in $\Cset_-$) for $x\in U$ satisfying $|x|<R$ with some $R>0$.
This yields part~(ii) along with \eqref{eqn:prop2a1}.

We turn to part (iii).
The ZS system  \eqref{eqn:ZS2} is rewritten as
\begin{equation}
w_\xi=\gamma_\xi(\xi)
\begin{pmatrix}
-ik & q(\gamma(\xi))\\
r(\gamma(\xi)) & ik
\end{pmatrix}w,
\label{eqn:prop2a2}
\end{equation}
on $\Gamma$.
Changing the independent variable as $\xi\mapsto k\xi$ in \eqref{eqn:prop2a2}, we have
\[
w_\xi=\gamma_\xi(\xi/k)
\begin{pmatrix}
-i & \epsilon q(\gamma(\xi/k))/k\\
r(\gamma(\xi/k))/k & i
\end{pmatrix}w,
\]
which reduces to
\[
w_\xi=
\begin{pmatrix}
-i & 0\\
0 & i
\end{pmatrix}w
\]
as $k\to\infty$.
This implies that
\begin{equation}
a(k),\bar{a}(k)\to 1,\quad
b(k),\bar{b}(k)\to 0\qquad
\mbox{as $|k|\to\infty$}.
\label{eqn:ab3}
\end{equation}
On the other hand, the ZS system \eqref{eqn:ZS2}
 has the Jost solutions satisfying \eqref{eqn:bc} for $k=0$,
 so that the scattering coefficients are still defined and analytic at $k=0$,
 as stated in Remark~4.2(i) of \cite{Y23a}.
So we apply the identity theorem (e.g., Theorem~3.2.6 of \cite{AF03})
 to obtain part~(iii), since $a(k)$ and $\bar{a}(k)$ are analytic
 in neighborhoods of $\Cset_+\cup\Rset$ and $\Cset_-\cup\Rset$, respectively.
\end{proof}

\begin{prop}
\label{prop:2b}
Let $k\in\Cset^\ast$.
We have the following$\,:$
\begin{itemize}
\setlength{\leftskip}{-1.6em}
\item[(i)]
If $r(x)=q(x)$, then $\bar{a}(k)=a(-k)$ and $\bar{b}(k)=b(-k);$
\item[(ii)]
If $r(x)=-q(x)$, then $\bar{a}(k)=a(-k)$ and $\bar{b}(k)=-b(-k);$
\item[(iii)]
If $r(x)=q(x)^\ast$, then $\bar{a}(k)=a(k^\ast)^\ast$ and $\bar{b}(k)=b(k^\ast)^\ast;$
\item[(iv)]
If $r(x)=-q(x)^\ast$, then $\bar{a}(k)=a(k^\ast)^\ast$ and $\bar{b}(k)=-b(k^\ast)^\ast$.
\end{itemize}
\end{prop}

\begin{proof}
Assume that $r(x)=q(x)$.
We rewrite \eqref{eqn:ZS2} as
\[
\begin{pmatrix}
w_{2x}\\
w_{1x}
\end{pmatrix}=
\begin{pmatrix}
ik & q(x)\\
q(x) & -ik
\end{pmatrix}\begin{pmatrix}
w_2\\
w_1
\end{pmatrix},
\]
which becomes the same as \eqref{eqn:ZS2} when replacing $k$ with $-k$,
 where $w_j$ represents the $j$th component of $w$ for $j=1,2$.
Since by \eqref{eqn:bc}
\begin{equation}
\begin{split}
&
\bar{\phi}_1(x;-k)\sim 0,\quad
\bar{\phi}_2(x;-k)\sim e^{-ikx}\quad
\mbox{as $x\to-\infty$},\\
&
\bar{\psi}_1(x;-k)\sim e^{ikx},\quad
\bar{\psi}_2(x;-k)\sim 0\quad
\mbox{as $x\to+\infty$},
\end{split}
\label{eqn:prop2b1}
\end{equation}
we have
\begin{align*}
&
\bar{\phi}_1(x;-k)=\phi_2(x;k),\quad
\bar{\phi}_2(x;-k)=\phi_1(x;k),\\
&
\bar{\psi}_1(x;-k)=\psi_2(x;k),\quad
\bar{\psi}_2(x;-k)=\psi_1(x;k),
\end{align*}
where $\phi_j$, $\bar{\phi}_j$, $\psi_j$ and $\bar{\psi}_j$
 represent the $j$th components of $\phi$, $\bar{\phi}$, $\psi$ and $\bar{\psi}$,
 respectively, for $j=1,2$.
Hence, by \eqref{eqn:ab}
\begin{align*}
\begin{pmatrix}
\phi_1(x;-k)\\
\phi_2(x;-k)
\end{pmatrix}
=&\begin{pmatrix}
\bar{\phi}_2(x;k)\\
\bar{\phi}_1(x;k)
\end{pmatrix}
=\bar{a}(k)
\begin{pmatrix}
\psi_2(x;k)\\
\psi_1(x;k)
\end{pmatrix}
+\bar{b}(k)
\begin{pmatrix}
\bar{\psi}_2(x;k)\\
\bar{\psi}_1(x;k)
\end{pmatrix}\\
=&
\bar{a}(k)
\begin{pmatrix}
\bar{\psi}_1(x;-k)\\
\bar{\psi}_2(x;-k)
\end{pmatrix}
+\bar{b}(k)
\begin{pmatrix}
\psi_1(x;-k)\\
\psi_2(x;-k)
\end{pmatrix},
\end{align*}
which yields part~(i).

Assume that $r(x)=-q(x)$.
We rewrite \eqref{eqn:ZS2} as
\[
\begin{pmatrix}
w_{2x}\\
w_{1x}
\end{pmatrix}=
\begin{pmatrix}
ik & -q(x)\\
q(x) & -ik
\end{pmatrix}\begin{pmatrix}
w_2\\
w_1
\end{pmatrix},
\]
which becomes the same as \eqref{eqn:ZS2}
 when replacing $k$ and either $w_1$ or $w_2$ with $-k$ and $-w_1$ or $-w_2$.
From \eqref{eqn:prop2b1} we have
\begin{align*}
&
-\bar{\phi}_1(x;-k)=\phi_2(x;k),\quad
\bar{\phi}_2(x;-k)=\phi_1(x;k),\\
&
\bar{\psi}_1(x;-k)=\psi_2(x;k),\quad
-\bar{\psi}_2(x;-k)=\psi_1(x;k),
\end{align*}
which yields part~(ii) by \eqref{eqn:ab}, as in part~(i).

Assume that $r(x)=q(x)^\ast$.
We rewrite \eqref{eqn:ZS2} as
\[
\begin{pmatrix}
w_{2x}\\
w_{1x}
\end{pmatrix}=
\begin{pmatrix}
ik & q(x)^\ast\\
q(x) & -ik
\end{pmatrix}\begin{pmatrix}
w_2\\
w_1
\end{pmatrix},
\]
which becomes the same as \eqref{eqn:ZS2}
 when taking complex conjugate for both sides and replacing $k^\ast$ with $k$.
Since by \eqref{eqn:bc}
\begin{equation}
\begin{split}
&
\bar{\phi}_1(x;k)^\ast\sim 0,\quad
\bar{\phi}_2(x;k)^\ast\sim e^{-ik^\ast x}\quad
\mbox{as $x\to-\infty$},\\
&
\bar{\psi}_1(x;k)^\ast\sim e^{ik^\ast x},\quad
\bar{\psi}_2(x;k)^\ast\sim 0\quad
\mbox{as $x\to+\infty$},
\end{split}
\label{eqn:prop2b2}
\end{equation}
we have
\begin{align*}
&
\bar{\phi}_1(x;k)^\ast=\phi_2(x;k^\ast),\quad
\bar{\phi}_2(x;k)^\ast=\phi_1(x;k^\ast),\\
&
\bar{\psi}_1(x;k)^\ast=\psi_2(x;k^\ast),\quad
\bar{\psi}_2(x;k)^\ast=\psi_1(x;k^\ast),
\end{align*}
which yields part~(iii) by \eqref{eqn:ab}, as in part~(i).

Assume that $r(x)=-q(x)^\ast$.
We rewrite \eqref{eqn:ZS2} as
\[
\begin{pmatrix}
w_{2x}\\
w_{1x}
\end{pmatrix}=
\begin{pmatrix}
ik & -q(x)^\ast\\
q(x) & -ik
\end{pmatrix}\begin{pmatrix}
w_2\\
w_1
\end{pmatrix},
\]
which becomes the same as \eqref{eqn:ZS2}
 when taking complex conjugate for both sides
 and replacing $k$ and either $w_1$ or $w_2$ with $k^\ast$ and $-w_1$ or $-w_2$. 
From \eqref{eqn:prop2b2} we have
\begin{align*}
&
-\bar{\phi}_1(x;k)^\ast=\phi_2(x;k^\ast),\quad
\bar{\phi}_2(x;k)^\ast=\phi_1(x;k^\ast),\\
&
\bar{\psi}_1(x;k)^\ast=\psi_2(x;k^\ast),\quad
-\bar{\psi}_2(x;k)^\ast=\psi_1(x;k^\ast),
\end{align*}
which yields part~(iv) by \eqref{eqn:ab}, as in part~(i).
Thus, we complete the proof.
\end{proof}

Let $G$ be the differential Galois group of \eqref{eqn:ZS2}.
Then $G$ is an algebraic group such that $G\subset\SL(2,\Cset)$ by \eqref{eqn:W}
 (see also Section~3 of \cite{Y23a}).
We have the following classification
 for such algebraic groups (see Section~2.1 of \cite{M99} for a proof).
Recall that an algebraic group $\G$ contains in general a unique maximal algebraic subgroup $\G^0$,
 which is called the \emph{connected component of the identity}
 or \emph{connected identity component}.

\begin{prop}
\label{prop:2c}
Any algebraic group $\G\subset\SL(2,\Cset)$ is similar to one of the following types:
\begin{enumerate}
\setlength{\leftskip}{-1.2em}
\item[(i)] $\G$ is finite and $\G^0= \{\id\}$,
 where $\id$ is the $2\times 2$  identity matrix$;$
\item[(ii)] $\G = \left\{
\begin{pmatrix}
\lambda & 0\\
\mu & \lambda^{-1} 
\end{pmatrix}
\middle|\,\lambda\text{ is a root of $1$, $\mu\in\Cset$}
\right\}$
and $\G^0 = \left\{\begin{pmatrix}
1&0\\
\mu & 1 
\end{pmatrix}
\middle|\,\mu \in \Cset\right\}$;
\item[(iii)] 
$\G = \G^0 = 
\left\{
\begin{pmatrix}
\lambda &0 \\
0 & \lambda ^{-1}
\end{pmatrix}
\middle|\,\lambda \in \Cset^{*}
\right\}$;
\item[(iv)]
$\G = \left\{
\begin{pmatrix}
\lambda & 0\\
0 & \lambda^{-1} 
\end{pmatrix},
\begin{pmatrix}
0 & -\beta^{-1}\\
\beta & 0 
\end{pmatrix}
\middle|\,\lambda, \beta \in \Cset^{*}
\right\}$
and $\G^0 = \left\{\begin{pmatrix}
\lambda &0\\
0 & \lambda^{-1}
\end{pmatrix}
\middle|\, \lambda \in \Cset^*\right\}$;
\item[(v)] $\G = \G^0 = \left\{
\begin{pmatrix}
\lambda & 0\\
\mu & \lambda^{-1}
\end{pmatrix}
\middle|\,\lambda \in \Cset^{*},\, \mu \in \Cset
\right\}$;
\item[(vi)] $\G = \G^0 = \SL (2,\, \Cset)$.
\end{enumerate}
\end{prop}

The system~\eqref{eqn:ZS2} is solvable by quadrature
 unless $G$ is of type (vi).
Proposition~\ref{prop:2c} plays a key role
 in the proof of Theorem~\ref{thm:main2} in Section~4.


\section{Proof of Theorem~\ref{thm:main1}}

In this section we give a proof of Theorem~\ref{thm:main1}.
We extend arguments used in the proof of Theorem~1.3 of \cite{Y23b}
 for the linear Schr\"odinger equations.
Henceforth we assume that the hypotheses of Theorem~\ref{thm:main1} hold
 and $b(k),\bar{b}(k)=0$ for $k\in\Rset^\ast$.

\begin{proof}[Proof of Theorem~$\ref{thm:main1}$]
We see by Proposition~\ref{prop:2a}(iii) and \eqref{eqn:2b}
 that $a(k),\bar{a}(k)\neq 0$ for $k\in\Rset^\ast$
 and $a(k)$ and $\bar{a}(k)$ have finitely many zeros
 in $\Cset_+$ and $\Cset_-$, respectively.
Assume that $a(k)$ and $\bar{a}(k)$ have $n$ and $\bar{n}$ zeros
 in $\Cset_+$ and $\Cset_-$, respectively, and
 let $k_j$ and $\bar{k}_j$ be their zeros of multiplicity $\nu_j$and $\bar{\nu}_j$  in $\Cset_+$
 for $j=1,\ldots,n$ and $j=1,\ldots,\bar{n}$, respectively.
Let the curve $\Gamma=\{\gamma(\xi)\in\Cset_+\mid \xi\in\Rset\}$
 contain no poles of $q(x),r(x)$,
 and let $U$ be a neighborhood of $\Gamma$ in $\Cset_+$ containing no poles of $q(x),r(x)$,
 as in the proof of Proposition~\ref{prop:2a}.

Let
\begin{equation}
\begin{split}
&
M(x;k)=\phi(x;k)e^{ikx},\quad
\bar{M}(x;k)=\bar{\phi}(x;k)e^{-ikx},\\
&
N(x;k)=\psi(x;k)e^{-ikx},\quad
\bar{N}(x;k)=\bar{\psi}(x;k)e^{i kx},
\end{split}
\label{eqn:MN2}
\end{equation}
By \eqref{eqn:bc} we have
\begin{equation}
\begin{split}
&
M(x;k)\to
\begin{pmatrix}
1\\
0
\end{pmatrix},\quad
\bar{M}(x;k)\to
\begin{pmatrix}
0\\
1
\end{pmatrix}
\quad\mbox{as $x\to-\infty$,}\\
&
N(x;k)\to
\begin{pmatrix}
0\\
1
\end{pmatrix},\quad
\bar{N}(x;k)\to
\begin{pmatrix}
1\\
0
\end{pmatrix}
\quad\mbox{as $x\to+\infty$.}
\end{split}
\label{eqn:MN2a}
\end{equation}
Since $\phi(x;k),\psi(x;k)$ (resp. $\bar{\phi}(x;k),\bar{\psi}(x;k)$) are analytic in $x\in U$
 and $k\in\Cset_+$ (resp. $k\in\Cset_-$),
 so are $M(x;k),N(x;k)$ (resp. $\bar{M}(x;k),\bar{N}(x;k)$).
It follows from \eqref{eqn:ab} that
\begin{equation}
\begin{split}
&
M(x;k)=a(k)\bar{N}(x;k)+b(k)N(x;k)e^{2ikx},\\
&
\bar{M}(x;k)=\bar{a}(k)N(x;k)+\bar{b}(k)\bar{N}(x;k)e^{-2ikx}.
\end{split}
\label{eqn:MN2b}
\end{equation}
Let
\begin{align*}
N_j^r(x):=\frac{\partial^r N}{\partial k^r}(x;k_j),\quad
 r=0,\ldots,\nu_j-1,\ j=1,\ldots,n,\\
\bar{N}_j^r(x):=\frac{\partial^r\bar{N}}{\partial k^r}(x;\bar{k}_j),\quad
 r=0,\ldots,\bar{\nu}_j-1,\ j=1,\ldots,\bar{n}.
\end{align*}
Differentiating \eqref{eqn:MN2b} with respect to $k$ and substituting $k=k_j$ or $\bar{k}_j$,
 we obtain
\begin{equation}
\begin{split}
&
\frac{\partial^r M}{\partial k^r}(x;k_j)
=\frac{\partial^r}{\partial k^r}(b(k)N(x;k)e^{2ikx})\Big|_{k=k_j},\quad
 r=0,\ldots,\nu_j-1,\\
&
\frac{\partial^r\bar{M}}{\partial k^r}(x;\bar{k}_j)
 =\frac{\partial^r}{\partial k^r}(\bar{b}(k)\bar{N}(x;k)e^{-2ikx})\Big|_{k=\bar{k}_j},\quad
 r=0,\ldots,\bar{\nu}_j-1,
\end{split}
\label{eqn:MN2c}
\end{equation}
for $j=1,\ldots,n$
 since the zeros $k_j$ and $\bar{k}_j$ of $a(k)$ and $\bar{a}(k)$
 are of multiplicity $\nu_j$ and $\bar{\nu}_j$, respectively. 
In particular, the right hand sides of the first and second equations in \eqref{eqn:MN2c}
 can be, respectively, represented by linear combinations of $N_j^r(x)$, $r=0,\ldots,\nu_j-1$,
 and $\bar{N}_j^r(x)$, $r=0,\ldots,\bar{\nu}_j-1$,
 where the coefficients are given by polynomials of $x$ and exponential functions
 $\{e^{2ik_jx}\}_{j=1}^n$ and $\{e^{-2i\bar{k}_jx}\}_{j=1}^{\bar{n}}$.
Actually, when $\nu_j,\bar{\nu}_j>1$, we have
\begin{align*}
&
M(x;k_j)
=b(k_j)e^{2ik_jx}N_j^0(x),\quad
\bar{M}(x;k_j)
=\bar{b}(\bar{k}_j)e^{-2i\bar{k}_jx}\bar{N}_j^0(x),\\
&
\frac{\partial M}{\partial k}(x;\bar{k}_j)
=b(k_j)e^{2ik_jx}N_j^1(x)+(b_k(k_j)+2ixb(k_j))e^{2ik_jx}N_j^0(x),\\
&
\frac{\partial\bar{M}}{\partial k}(x;\bar{k}_j)
=\bar{b}(\bar{k}_j)e^{-2i\bar{k}_jx}\bar{N}_j^1(x)
 +(\bar{b}_k(\bar{k}_j)-2ix\bar{b}(\bar{k}_j))e^{-2i\bar{k}_jx}\bar{N}_j^0(x),\\
 & \qquad\ldots.
\end{align*}

Define the projection operators
\[
\P^\pm(f)=\frac{1}{2\pi i}\int_{-\infty}^\infty\frac{f(\kappa)}{\kappa-(k\pm i 0)}\d\kappa
\]
for $f\in L^1(\Rset)$, where $k\in\Cset_+$ or $\Cset_-$
 depending on whether the upper or lower signs are taken simultaneously.
If $f_+$ (resp. $f_-$) is analytic in $\Cset_+$ (resp. in $\Cset_-$)
 and $f_\pm(k)\to 0$ as $|k|\to\infty$, then
\[
\P^\pm(f_\pm)=\pm f_\pm,\quad
\P^\pm(f_\mp)=0
\]
(see, e.g., Chapter~7 of \cite{AF03}).
Dividing both sides of the first and second equations of \eqref{eqn:MN2b}
 by $a(k)$ and $\bar{a}(k)$,
 we apply the projection operators $\P^-$ and $\P^+$ to the resulting equations
 and obtain
\begin{equation}
\begin{split}
&
\bar{N}(x;k)=
\begin{pmatrix}
1\\
0
\end{pmatrix}
+\frac{1}{2\pi i}\sum_{j=1}^{n}\int_{|\kappa-k_j|=\delta}
 \frac{M(x;\kappa)}{(k-\kappa)a(\kappa)}\d\kappa,\\
&
N(x;k)=
\begin{pmatrix}
0\\
1
\end{pmatrix}
+\frac{1}{2\pi i}\sum_{j=1}^{\bar{n}}\int_{|\kappa-\bar{k}_j|=\delta}
 \frac{\bar{M}(x;\kappa)}{(k-\kappa)\bar{a}(\kappa)}\d\kappa,
\end{split}
\label{eqn:MN2d}
\end{equation}
where $\delta>0$ is sufficiently small
 and the asymptotic relations,
\[
M(x;k),\bar{N}(x;k)\to
\begin{pmatrix}
1\\
0
\end{pmatrix},\quad
\bar{M}(x;k),N(x;k)\to
\begin{pmatrix}
0\\
1
\end{pmatrix}
\]
as $k\to\infty$ (see, e.g.,  Section 2.2.2 of \cite{APT04}), have been used,
 since $M(x;k),N(x;k)$ and $\bar{M}(x;k),\bar{N}(x;k)$ are analytic
 in $k\in\Cset_+$ and $\Cset_-$, respectively,
 and $b(k),\bar{b}(k)=0$ on $\Rset^\ast$.
Using the method of residues, we see that
 the integrals in the first and second equations of \eqref{eqn:MN2d}
 are, respectively, represented by linear combinations
 of $N_j^r(x)$ and $\bar{N}_j^r(x)$, $r=0,\ldots,\nu_j-1$ or $\bar{\nu}_j-1$,
 $j=1,\ldots,n$ or $\bar{n}$,
 where the coefficients are given by polynomials
 of $x$, $\{e^{2ik_jx}\}_{j=1}^n$ and $\{e^{-2i\bar{k}_jx}\}_{j=1}^{\bar{n}}$,
 like \eqref{eqn:MN2c}.
For example, we have
\begin{align*}
&
\frac{1}{2\pi i}\int_{|\kappa-k_j|=\delta}
 \frac{M(x;\kappa)}{(k-\kappa)a(\kappa)}\d\kappa
 =\frac{e^{2ik_j x}}{k-k_j}\frac{b(k_j)}{a_k(k_j)}N_j^0(x),\\
&
\frac{1}{2\pi i}\int_{|\kappa-\bar{k}_j|=\delta}
 \frac{\bar{M}(x;\kappa)}{(k-\kappa)\bar{a}(\kappa)}\d\kappa
 =\frac{e^{-2i\bar{k}_j x}}{k-\bar{k}_j}\frac{\bar{b}(\bar{k}_j)}{\bar{a}_k(\bar{k}_j)}
  \bar{N}_j^0(x)
\end{align*}
when $\nu_j,\bar{\nu}_j=1$, and
\begin{align*}
&
\frac{1}{2\pi i}\int_{|\kappa-k_j|=\delta}
 \frac{M(x;\kappa)}{(k-\kappa)a(\kappa)}\d\kappa\\
&
=\frac{2e^{2ik_j x}}{(k-k_j)a_{kk}(k_j)}\biggl(b(k_j)N_j^1(x)\\
&\quad
+\biggl(b_k(k_j)+2ixb(k_j)
 +\frac{b(k_j)}{k-k_j}-\frac{a_{kkk}(k_j)b(k_j)}{3a_{kk}(k_j)}\biggr)
 N_j^0(x) \biggr),\\
&
\frac{1}{2\pi i}\int_{|\kappa-\bar{k}_j|=\delta}
 \frac{\bar{M}(x;\kappa)}{(k-\kappa)\bar{a}(\kappa)}\d\kappa\\
&
=\frac{2e^{-2i\bar{k}_j x}}{(k-\bar{k}_j)\bar{a}_{kk}(\bar{k}_j)}
 \biggl(\bar{b}(\bar{k}_j)\bar{N}_j^1(x)\\
 &\quad
 +\biggl(\bar{b}_k(\bar{k}_j)-2ix\bar{b}(\bar{k}_j)
 +\frac{\bar{b}(\bar{k}_j)}{k-\bar{k}_j}
 -\frac{\bar{a}_{kkk}(\bar{k}_j)\bar{b}(k_j)}{3\bar{a}_{kk}(\bar{\bar{k}}_j)}\biggr)\bar{N}_j^0(x)
 \biggr)
\end{align*}
when $\nu_j,\bar{\nu}_j=2$.

Differentiating \eqref{eqn:MN2d} with respect to $k$
 up to $\bar{\nu}_j-1$ or  $\nu_j-1$ times
 and setting $k=\bar{k}_j$ or $k_j$,
 we obtain a system of linear equations about $N_j^r(x)$ and $\bar{N}_j^r(x)$.
 $r=0,\ldots,\nu_j-1$ or $\bar{\nu}_j-1$, $j=1,\ldots,n$ or $\bar{n}$.
Note that
\[
\frac{\partial}{\partial k}(N(x;k)e^{-2ik x})
=\left(\frac{\partial N}{\partial k}(x;k)-2ixN(x;k)\right)e^{-2ik x}
\]
and so on.
We solve the system of linear equations to express them as rational functions
 of $x$, $\{e^{2ik_jx}\}_{j=1}^n$ and $\{e^{2i\bar{k}_jx}\}_{j=1}^{\bar{n}}$,
 using basic arithmetic operations.
Hence, we see by \eqref{eqn:MN2} and \eqref{eqn:MN2d} that
 the Jost solutions $\psi(x;k),\bar{\psi}(x;k)$
 are also represented by rational functions
 of $x$, $\{e^{ik_jx}\}_{j=1}^n$ and $\{e^{i\bar{k}_jx}\}_{j=1}^{\bar{n}}$.
This means the desired result.
 \end{proof}
 
\begin{rmk}
We show that
\begin{equation}
q(x)=2i\lim_{|k|\to\infty}kN_1(x;k),\quad
r(x)=-2i\lim_{|k|\to\infty}k\bar{N}_2(x;k),
\label{eqn:qr2}
\end{equation}
where $N_\ell(x;k)$ and $\bar{N}_\ell(x;k)$
 are the $\ell$th components of $N(x;k)$ and $\bar{N}(x;k)$, respectively, for $\ell=1,2$
$($see, e.g., Section~$2.2.2$ of {\rm\cite{APT04})},
Hence, $q(x)$ and $r(x)$ are represented by rational functions
 of $x$, $\{e^{2ik_jx}\}_{j=1}^n$ and $\{e^{2i\bar{k}_jx}\}_{j=1}^{\bar{n}}$.
Moreover, it follows from \eqref{eqn:ab3}, \eqref{eqn:MN2a}, \eqref{eqn:MN2d} and \eqref{eqn:qr2}
 that
\[
q(x),r(x)\to 0\quad\mbox{as $x\to\pm\infty$}.
\]
See also Section~$4.2.2$ of {\rm\cite{Y23a}}.
\end{rmk}

\begin{ex}
\label{ex:3a}
Let $n,\bar{n}=1$ and $\nu_1,\bar{\nu}_1=2$ in the proof of Theorem~$\ref{thm:main1}$.
Assume that $k_1=i$, $\bar{k}_1=-i$,
 $a_{kkk}(k_1),\bar{a}_{kkk}(\bar{k}_1),b_k(k_1),\bar{b}_k(\bar{k}_1)=0$ and
\begin{equation}
a_{kk}(k_1),\bar{a}_{kk}(\bar{k}_1)=-\tfrac{1}{2},\quad
b(k_1),\bar{b}(\bar{k}_1)=1.
\label{eqn:ex3a1}
\end{equation}
Then the system of linear equations about $N_1^j(x),\bar{N}_1^j(x)$, $j=0,1$,
 discussed at the end of the proof of Theorem~$\ref{thm:main1}$ is given by
\begin{align*}
\bar{N}_1^0(x)
=&\begin{pmatrix}
 1\\
 0
\end{pmatrix}
+e^{-2x}\left(iN_1^1(x)-\left(2x+\tfrac{1}{2}\right)N_1^0(x)\right),\\
N_1^0(x)
=&\begin{pmatrix}
 0\\
 1
\end{pmatrix}
-e^{-2x}\left(i\bar{N}_1^1(x)+\left(2x+\tfrac{1}{2}\right)\bar{N}_1^0(x)\right),\\
\bar{N}_1^1(x)
=& \frac{e^{-2x}}{2}\left(N_1^1(x)+i(2x+1)N_1^0(x)\right),\\
N_1^1(x)
=&
\frac{e^{-2x}}{2}\left(\bar{N}_1^1(x)-i(2x+1)\bar{N}_1^0(x)\right).
\end{align*}
Solving the above equation, we obtain
\begin{align*}
&
N_1^0(x)=\frac{e^{2x}}{d(x)}
\begin{pmatrix}
(4x+1)e^{4x}+1\\
e^{2x}(e^{4x}-4x-3)
\end{pmatrix},\quad
\bar{N}_1^0(x)=\frac{2e^{2x}}{d(x)}
\begin{pmatrix}
e^{2x}(e^{4x}-4x-3)\\
((4x+1)e^{4x}+1)
\end{pmatrix},\\
&
N_1^1(x)=\frac{ie^{2x}}{d(x)}
\begin{pmatrix}
(2x+1)(e^{4x}-1)\\
2e^{2x}(2x+1)^2
\end{pmatrix},\quad
\bar{N}_1^1(x)=-\frac{ie^{2x}}{d(x)}
\begin{pmatrix}
2e^{2x}(2x+1)^2\\
(2x+1)(e^{4x}-1)
\end{pmatrix},
\end{align*}
where
\[
d(x)=e^{8x}-2(8x^2+8x+3)e^{4x}+1.
\]
Using \eqref{eqn:MN2d}, we have
\begin{align*}
&
\bar{N}(x;k)\\
&
=\begin{pmatrix}
1\\
0
\end{pmatrix}+
\frac{4i}{(k-i)^2d(x)}
\begin{pmatrix}
(8(k-i)x^2+4(k-2i)x+k-2i)e^{4x}+k\\
-e^{2x}((2(k-i)x-i)e^{4x}+2(k+i)x+2k+i)
\end{pmatrix},\\
&
N(x;k)\\
&
=\begin{pmatrix}
0\\
1
\end{pmatrix}+
\frac{4i}{(k+i)^2d(x)}
\begin{pmatrix}
e^{2x}((2(k+i)x+i)e^{4x}+2(k-i)x+2k-i)\\
-(8(k+i)x^2+4(k+2i)x+k+2i)e^{4x}-k
\end{pmatrix},
\end{align*}
from which the Jost solutions $\psi(x;k),\bar{\psi}(x;k)$
 are obtained by \eqref{eqn:MN2} as
\[
\psi(x;k)=N(x;k)e^{ikx},\quad
\bar{\psi}(x;k)=\bar{N}(x;k)e^{-ikx}.
\]
It follows from \eqref{eqn:qr2} that
\begin{equation}
q(x)=r(x)=\frac{32e^{2x}(4xe^{4x}+x+1)}{d(x)}.
\label{eqn:ex3a2}
\end{equation}
See Fig.~$\ref{fig:3a}$ for the shape of $q(x),r(x)$.
It has simple poles at $x=-0.245036\ldots$ and $0.864558\ldots$.
The potential \eqref{eqn:ex3a2} is of a negaton-type
 {\rm\cite{RSK96,S92,WF07,WZF08}}.

\begin{figure}
\includegraphics[scale=0.7]{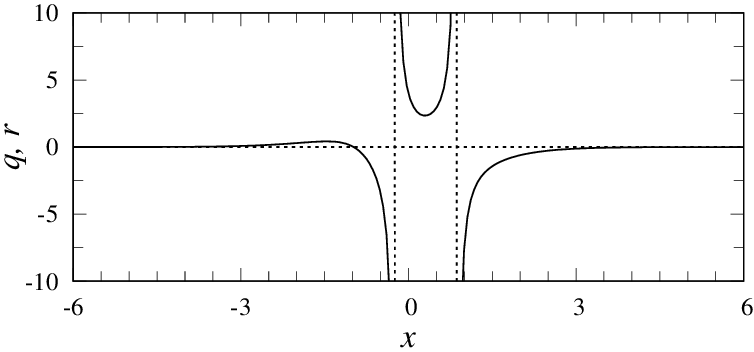}
\caption{Potentials $q(x),r(x)$ in Example~\ref{ex:3a}:
They have a simple pole at $x=-0.245036\ldots$ and  and $0.864558\ldots$,
 the loci of which are represented by vertical dotted lines.
\label{fig:3a}}
\end{figure}

Let $k\in\Rset^\ast$.
We see that
\[
\psi(x;k)\sim\frac{(k-i)^2}{(k+i)^2}
\begin{pmatrix}
0\\
1
\end{pmatrix}e^{ikx},\quad
\bar{\psi}(x;k)\sim\frac{(k+i)^2}{(k-i)^2}
\begin{pmatrix}
1\\
0
\end{pmatrix}e^{-ikx}
\]
as $x\to -\infty$.
Hence, by \eqref{eqn:bc} we have
\[
\phi(x;k)=\frac{(k-i)^2}{(k+i)^2}\bar{\psi}(x;k),\quad
\bar{\phi}(x;k)=\frac{(k+i)^2}{(k-i)^2}\psi(x;k),
\]
so that
\[
a(k)=\frac{(k-i)^2}{(k+i)^2},\quad
\bar{a}(k)=\frac{(k+i)^2}{(k-i)^2}.
\]
On the other hand, we see that
\[
\psi(x;i)e^{-x}\sim
\begin{pmatrix}
1\\
0
\end{pmatrix}\quad\mbox{and}\quad
\bar{\psi}(x;-i)e^{-x}\sim
\begin{pmatrix}
0\\
1
\end{pmatrix}
\]
as $x\to -\infty$, and consequently
\[
\phi(x;i)=\psi(x;i),\quad
b(i)=1
\]
and
\[
\bar{\phi}(x;-i)=\bar{\psi}(x;-i),\quad
\bar{b}(-i)=1.
\]
Moreover,
\[
a_{k}(i),\bar{a}_{k}(-i)=0,\quad
a_{kk}(i),\bar{a}_{kk}(-i)=-\tfrac{1}{2}.
\]
These agree with our assumption \eqref{eqn:ex3a1}.
\end{ex}


\section{Proof of Theorem~\ref{thm:main2}}
In this section we give a proof of Theorem~\ref{thm:main2}.
We extend arguments used in the proof of Theorem~1.4 of \cite{Y23b}
 for the linear Schr\"odinger equations, as in Section~3.
Henceforth we assume that $k\in\Rset^\ast$
 and the hypotheses of Theorem~\ref{thm:main2} hold.

Letting $y=1/x$, we rewrite \eqref{eqn:ZS2} as
\begin{equation}
w_y=-\frac{1}{y^2}
\begin{pmatrix}
-ik & \bar{q}(y)\\
\bar{r}(y) & ik
\end{pmatrix}w,
\label{eqn:v}
\end{equation}
where $\tilde{q}(y)=q(1/y)$ and $\tilde{r}(y)=r(1/y)$.
As shown below (see Lemma~\ref{prop:4a}),
 Eq.~\eqref{eqn:v} has an irregular singularity at $y=0$,
 i.e., Eq.~\eqref{eqn:ZS2} has an irregular singularity at infinity $x=\infty$.
To obtain a necessary condition for the integrabilty of \eqref{eqn:ZS2},
 we compute its formal monodromy, exponential torus and Stokes matrices
 \cite{MS16,PS03,S09} around $y=0$ for \eqref{eqn:v}.
See these references for necessary information on the following computations.
A concise review of the theory was also given in Appendix~A of \cite{Y23b}.

Letting
\[
w=\begin{pmatrix}
e^{-ik/y} & 0\\
0 & e^{ik/y}
\end{pmatrix}\eta,
\]
we rewrite \eqref{eqn:v} as
\begin{equation*}
\eta_y=-\frac{1}{y^2}
\begin{pmatrix}
0 & e^{2ik/y}\bar{q}(y)\\
e^{-2ik/y}\bar{r}(y) & 0
\end{pmatrix}\eta,
\end{equation*}
or equivalently,
\begin{equation}
\begin{split}
&
\eta_{1yy}+\left(\frac{2}{y}+\frac{2ik}{y^2}-\frac{\bar{q}_y(y)}{\bar{q}(y)}\right)\eta_{1y}
-\frac{\bar{q}(y)\bar{r}(y)}{y^4}\eta_1=0,\\
&
\eta_{2yy}+\left(\frac{2}{y}-\frac{2ik}{y^2}-\frac{\bar{r}_y(y)}{\bar{r}(y)}\right)\eta_{2y}
-\frac{\bar{q}(y)\bar{r}(y)}{y^4}\eta_2=0,
\end{split}
\label{eqn:eta}
\end{equation}
where $\eta_j$ represents the $j$th component of $\eta$ for $j=1,2$.
Since $q(x),r(x)$ satisfy \eqref{eqn:A1} and are analytic near $x=\infty$, we can write
\begin{align*}
\bar{q}(y)=q_0y^2+q_1y^3+q_2y^4+q_3y^5+\ldots,\quad
\bar{r}(y)=r_0y^2+\ldots,
\end{align*}
where $q_j,r_0\in\Cset$, $j=0,1,2,3$, are constants.
Hence, the coefficients of the first-order derivatives $\eta_{jy}$, $j=1,2$,
 have a pole of second order at $y=0$,
 so that by a standard result on higher-order scalar linear differential equations
 (see, e.g., Theorem~7 of \cite{B00}),
 both equations in \eqref{eqn:eta} have an irregular singularity at $y=0$.
This immediately  yields the following on \eqref{eqn:v}.

\begin{lem}
\label{prop:4a}
Eq.~\eqref{eqn:v} has an irregular singularity at $y=0$.
\end{lem}

Letting $\zeta=\eta_{y}/\eta$ in the first equation of \eqref{eqn:eta},
 we have a Riccati equation
\begin{equation}
\zeta_{y}+\zeta^2
+\left(\frac{2}{y}+\frac{2ik}{y^2}-\frac{\bar{q}_y(y)}{\bar{q}(y)}\right)\zeta
 -\frac{\bar{q}(y)\bar{r}(y)}{y^4}=0.
\label{eqn:zeta}
\end{equation}
To obtain formal solutions to \eqref{eqn:zeta}, we write
\[
\zeta=\sum_{j=0}^\infty \zeta_{j}y^{j-2}.
\]
Substituting the above expression into \eqref{eqn:zeta}, we have
\begin{align*}
\zeta_0^2+2ik\zeta_0=0,\quad
2\zeta_0\zeta_{j}+2ik\zeta_j+s_j(\zeta_{0},\ldots,\zeta_{j-1})=0,\quad j\ge 1,
\end{align*}
where $s_j(\zeta_{0},\ldots,\zeta_{j-1})$ is a polynomial of $\zeta_{0},\ldots,\zeta_{j-1}$
 for $j\ge 1$.
In particular,
\begin{align*}
&
s_1(\zeta_0)=-2\zeta_0,\quad
s_2(\zeta_0,\zeta_1)=-\frac{q_1}{q_0}\zeta_0-\zeta_1+\zeta_1^2,\\
&
s_3(\zeta_0,\zeta_1,\zeta_2)
=\left(\frac{q_1^2}{q_0^2}-\frac{2q_2}{q_0}\right)\zeta_0
-\frac{q_1}{q_0}\zeta_1+2\zeta_1\zeta_2
\end{align*}
and
\begin{align*}
s_4(\zeta_0,\zeta_1,\zeta_2,\zeta_3)
=& -\left(\frac{q_1^3}{q_0^3}-\frac{3q_1q_2}{q_0^2}+\frac{3q_3}{q_0}\right)\zeta_0\\
&
+\left(\frac{q_1^2}{q_0^2}-\frac{2q_2}{q_0}\right)\zeta_1
+\frac{q_1}{q_0}\zeta_2+\zeta_2^2+\zeta_3+2\zeta_1\zeta_3-q_0r_0.
\end{align*}
Thus, we obtain two formal solutions to \eqref{eqn:zeta} as
\[
\zeta(y)=-\frac{2ik}{y^2}+\cdots,\quad
-\frac{ iq_0r_0}{2k}y^2+\cdots,
\]
which yield two formal solution to \eqref{eqn:v},
\begin{equation*}
w_1(y)=\exp\left(-\frac{ik}{y}\right)\bar{w}_{11}(y),\
 \exp\left(\frac{ik}{y}\right)\bar{w}_{12}(y),
\end{equation*}
where $\bar{w}_{1j}(y)$, $j=1,2$, are formal power series in $y$.
Hence, 
 we can express a formal fundamental matrix of \eqref{eqn:v} as
\begin{equation}
V(y)=\tilde{V}(y)e^{Q(y)},
\label{eqn:V}
\end{equation}
where $\tilde{V}(y)$ is a $2\times 2$ formal meromorphic invertible matrix of power series in $y$ and 
\begin{align*}
Q(y)=
\begin{pmatrix}
-ik/y & 0\\
0 & ik/y
\end{pmatrix}.
\end{align*}

On the other hand, using the Jost solutions,
 we have two fundamental matrices of \eqref{eqn:v},
\begin{equation}
\begin{pmatrix}
\phi(1/y;k), \bar{\phi}(1/y;k)
\end{pmatrix}\sim
\begin{pmatrix}
e^{-ik/y} & 0\\
0 & e^{ik/y}
\end{pmatrix}
\label{eqn:phi}
\end{equation}
as $y\to -0$ in a region containing $\Re y<0$, and
\begin{equation}
\begin{pmatrix}
\bar{\psi}(1/y;k),\psi(1/y;k)
\end{pmatrix}\sim
\begin{pmatrix}
e^{-ik/y} & 0\\
0 & e^{ik/y}
\end{pmatrix}
\label{eqn:psi}
\end{equation}
as $y\to+0$ in a region containing $\Re y>0$,
 where we have used \eqref{eqn:bc}.
These fundamental matrices have the form of \eqref{eqn:V}.

\begin{figure}
\includegraphics[scale=0.6]{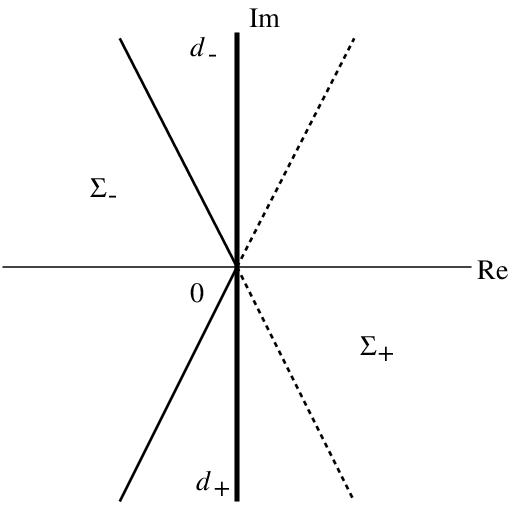}
\caption{Singular directions $d_\pm$ and sectors $\Sigma_\pm$:
The boundaries of $\Sigma_+$ and $\Sigma_-$ are plotted as solid and dotted lines..
\label{fig:1}}
\end{figure}

Substitution of $ye^{i\theta}$ into \eqref{eqn:V} yields
\[
V(ye^{i\theta})=\tilde{V}(ye^{i\theta})e^{Q(ye^{i\theta})},
\]
so that the \emph{formal monodromy} becomes
\[
\hat{M}=V(y)^{-1}V(ye^{2\pi i})=\id
\]
by $V(ye^{2\pi i})=V(y)\hat{M}$.
The \emph{exponential torus} is given by
\[
\T:=\mathrm{Gal}(\Cset((y))(e^{ik/y},e^{-ik/y})/\Cset((y)))
=\left\{
\begin{pmatrix}
c & 0\\
0 & c^{-1}
\end{pmatrix}
\bigg| c\in\Cset^\ast\right\},
\]
where $\Cset((y))$ denotes the field of the formal Laurent series over $\Cset$.
Here $\mathrm{Gal}(\Lset/\Kset)$ denotes the differential Galois group
 for a differential field extension $\Lset\supset\Kset$,
 where $\Kset$ is a differential field.
The \emph{Stokes directions} on which $\Re(\mp ik/y)=0$, are given by $\arg y=0,\pi$.

We specifically assume that $k>0$.
The case of $k<0$ can be treated similarly.
We have $\Re(-ik/y)<0$ and $\Re(ik/y)<0$, respectively.
 for $\arg y\in(0,\pi)$ and $(\pi,2\pi)$,
 which are called \emph{negative Stokes pairs}.
So the \emph{singular directions} $d_\mp$,
 which are bisectors of the negative Stokes pairs
 $(\tfrac{1}{2}\pi\mp\tfrac{1}{2}\pi,\tfrac{3}{2}\pi\mp\tfrac{1}{2}\pi)$,
 become $\arg y=\pi\mp\tfrac{1}{2}\pi$.
See Fig.~\ref{fig:1}.

Let $a(k),\bar{a}(k)\neq 0$.
Using an argument in Section~4 of \cite{Y23b}, we can write the Stokes matrices as
\[
S_-=\begin{pmatrix}
1 & \alpha_- \\
0 & 1
\end{pmatrix}
\quad\mbox{and}\quad
S_+=\begin{pmatrix}
1 & 0\\
\alpha_+ & 1
\end{pmatrix}
\]
for the singular directions $d_-$ and $d_+$, respectively,
 where $\alpha_\pm$ are constants which may be zero.
Moreover, with the assistance of \eqref{eqn:phi} and \eqref{eqn:psi},
 we can prove that if $\alpha_-$ and $\alpha_+=0$, then $b(-k)$ and $b(k)$, respectively,
 as in Lemma~4.1 of \cite{Y23b}.

Ramis' theorem \cite{L94,PS03,RM90} (see also Theorem~A.2 of \cite{Y23b})
 implies that the differential Galois group $G$ of \eqref{eqn:ZS2} contains
 the Zariski closure of a group generated by the formal monodromy $M=\id$,
 the exponential torus $\T$ and the Stokes matrices $S_\pm$.
Recall that $G\subset\SL(2,\Cset)$.
In particular, $G$ is not finite,
 i.e., cases (i) and (ii) of Proposition~\ref{prop:2c} do not occur,
 and obviously case (iv)  of Proposition~\ref{prop:2c} does not occur.
Hence, if Eq.~\eqref{eqn:ZS2} is integrable in the sense of differential Galois theory,
 then we have $b(k)$ or $\bar{b}(k)=0$
 since one of case~(iii) or (v) of Proposition~\ref{prop:2c} occurs.
So we have the following.

\begin{lem}
\label{prop:4b}
Let $k\in\Rset^\ast$ and assume that $a(k),\bar{a}(k)\neq 0$.
If Eq.~\eqref{eqn:ZS2} is solvable by quadrature, then $b(k)$ or $\bar{b}(k)=0$.
\end{lem}

\begin{proof}[Proof of Theorem~$\ref{thm:main2}$]
Suppose that Eq.~\eqref{eqn:ZS2} is solvable by quadrature for any $k\in\Rset^\ast$.
Since by Proposition~\ref{prop:2a}(iii) $a(k)$ or $\bar{a}(k)=0$ only at discrete points,
 it follows from Lemma~\ref{prop:4b} that $b(k)$ or $\bar{b}(k)=0$
 except at the discrete points.
Hence, by the identity theorem (e.g., Theorem~3.2.6 of \cite{AF03})
 and Proposition~\ref{prop:2b},
 we have $b(k),\bar{b}(k)=0$ for any $k\in\Rset^\ast$
 since by Proposition~\ref{prop:2a}(i) $b(k),\bar{b}(k)$ are analytic
 in a neighborhood of $\Rset^\ast$ in $\Cset^\ast$.
Thus, we complete the proof.
\end{proof}



\section{Special Rational Potentials}

Finally, we consider the case in which $q(x),r(x)$ are rational functions
 and satisfy \eqref{eqn:A1a} but do not \eqref{eqn:A1}.
Let the degrees of the numerator and denominator of $q(x)$
 be the same as those of $r(x)$ and denoted by $m_1$ and $m_2$, respectively.
In such a situation, the following result was proved
 for the linear Schr\"odinger equation \eqref{eqn:SE} in \cite{Y23b}.

\begin{thm}
\label{thm:KdV}
If $m_2-m_1=1$,
 then Eq.~\eqref{eqn:SE} is not solvable by quadrature for some $k\in\Rset^\ast$.
\end{thm}

Obviously, if $m_2-m_1>1$, then the rational functions $q(x),r(x)$ satisfy \eqref{eqn:A1}
 and condition~\textrm{(A${}_1$)}.
The proof of Theorem~\ref{thm:KdV} in \cite{Y23b} relied on Kovacic's algorithm \cite{K86},
 which determine the solvability of second-order linear differential equations of the form 
\[
v_{xx}=s(x)v
\]
by quadrature, where $s(x)$ is a rational function, and poles of $s(x)$,
 which do not depend on $k$ in the linear Schr\"odinger equation \eqref{eqn:SE},
 play a key role there.
It is difficult to obtain such a result as Theorem~\ref{thm:KdV} for the ZS system \eqref{eqn:ZS2}.
In the following we explain its reason. 

Letting
\[
w=\begin{pmatrix}
e^{-ikx} & 0\\
0 & e^{ikx}
\end{pmatrix}\eta,
\]
we rewrite \eqref{eqn:ZS2} as
\begin{equation*}
\eta_x=
\begin{pmatrix}
0 & e^{2ikx}q(x)\\
e^{-2ikx}r(x) & 0
\end{pmatrix}\eta,
\end{equation*}
or equivalently,
\begin{equation}
\begin{split}
&
\eta_{1xx}-\left(2ik+\frac{q_x(x)}{q(x)}\right)\eta_{1x}-q(x)r(x)\eta_1=0,\\
&
\eta_{2xx}+\left(2ik-\frac{r_x(x)}{r(x)}\right)\eta_{2x}-q(x)r(x)\eta_2=0,
\end{split}
\label{eqn:eta2}
\end{equation}
where $\eta_j$ represents the $j$th component of $\eta$ for $j=1,2$.
Letting
\[
\zeta=\eta_1\exp\left(-\tfrac{1}{2}\int\left(2ik+\frac{q_x(x)}{q(x)}\right)\d x\right)
\]
in the first equation of \eqref{eqn:eta2}, we have
\begin{equation}
\zeta_{xx}+\left(k^2-iku_1(x)+u_2(x)\right)\zeta=0,
\label{eqn:zeta2}
\end{equation}
where
\[
u_1(x)=\frac{q_x(x)}{q(x)},\quad
u_2(x)=\frac{q_{xx}(x)}{2q(x)}-\frac{3q_x(x)^2}{4q(x)^2}-q(x)r(x).
\]
Since its degrees of the numerator and denominator for $u_1(x)$
 are, respectively, $m_1+2m_2-1$ and $m_1+2m_2$,
 its order at $x=\infty$ is $(m_1+2m_2)-(m_1+2m_2-1)=1$ for any $m_1<m_2$,
 while the order of $u_2(x)$ at $x=\infty$ is $\min(2(m_2-m_1),2)>1$.
In particular, poles of the coefficient of $\zeta$ in \eqref{eqn:zeta2}
 depend on $k$ in general.
This yields a difficulty in getting a result similar to Theorem~\ref{thm:KdV}
 for the ZS system \eqref{eqn:ZS2}.
A  similar argument can be applied to the second equation of \eqref{eqn:eta2}.

\section*{Acknowledgements}
This work was partially supported by the JSPS KAKENHI Grant Number JP22H01138.

\end{document}